\documentclass[12pt,reqno]{amsart}

\usepackage{amssymb}

\newtheorem{theorem}{Theorem}[section]
\newtheorem{proposition}[theorem]{Proposition}
\newtheorem{lemma}[theorem]{Lemma}
\newtheorem{corollary}[theorem]{Corollary}

\theoremstyle{definition}
\newtheorem{definition}[equation]{Definition}

\numberwithin{equation}{section}

\begin{document}

\baselineskip=15.5pt

\title[Holomorphic affine connections on non-K\"ahler manifolds]{Holomorphic affine 
connections on non-K\"ahler manifolds}

\author[I. Biswas]{Indranil Biswas}

\address{School of Mathematics, Tata Institute of Fundamental
Research, Homi Bhabha Road, Bombay 400005, India}

\email{indranil@math.tifr.res.in}

\author[S. Dumitrescu]{Sorin Dumitrescu}

\address{Universit\'e Nice Sophia Antipolis, LJAD, UMR 7351 CNRS,
Parc Valrose, 
06108 Nice Cedex 2, France}

\email{dumitres@unice.fr}



\date{}

\begin{abstract}
Our aim here is to investigate the holomorphic geometric structures on
compact complex manifolds which may not be K\"ahler. 
We prove that holomorphic geometric structures of affine type on compact 
Calabi-Yau manifolds with polystable tangent bundle (with respect to some Gauduchon 
metric on it) are locally homogeneous. In particular, if the geometric structure is rigid in Gromov's sense, 
then the fundamental group of the manifold must be infinite. We also prove that 
compact complex manifolds of algebraic dimension one bearing a holomorphic Riemannian 
metric must have infinite fundamental group.
\end{abstract}

\maketitle

\section{Introduction}

A motivation of this article is the following open question: {\it If the holomorphic 
tangent bundle $TX$ of a compact complex manifold $X$ admits a holomorphic connection 
(also called holomorphic affine connection), does it also admit a flat holomorphic 
connection ?}

If $X$ is K\"ahler the classical Chern-Weil theory shows that Chern classes with 
rational coefficients $c_i(X,\mathbb{Q})$ must vanish and, consequently, there is no 
topological obstruction for $TX$ to admit a flat connection. Moreover, from Yau's 
theorem proving Calabi conjecture~\cite{Ya} (see also~\cite{Be} and ~\cite{IKO}) $X$ 
admits a finite unramified cover which is a complex torus. In particular,
the standard complex affine structure of the torus produces 
a flat torsionfree affine connection on $X$.

It is also proved in~\cite{IKO} that compact complex surfaces admitting holomorphic 
affine connections are biholomorphic to quotients of open sets in $\mathbb{C}^2$ by a 
discrete group of complex affine transformations acting properly and without fixed 
points. In particular, these surfaces also admit a torsionfree flat affine connection 
induced by the complex affine structure of $\mathbb{C}^2$.

Recall that an interesting class of compact non-K\"ahler manifolds which generalize 
complex tori are those manifolds whose holomorphic tangent bundle is holomorphically 
trivial. These,  called parallelizable manifolds, are by a result of Wang~\cite{Wa} 
known to be biholomorphic to a quotient of a complex Lie group $G$ by a cocompact 
lattice $\Gamma$ in $G$. Such a quotient $G/ \Gamma$ is K\"ahler if and only if $G$ is 
abelian.

All parallelizable manifolds admit an obvious flat affine connection for which all 
global holomorphic vector fields coming from right invariant vector fields on $G$ are 
parallel. It was proved in~\cite{D5} that parallelizable manifolds $G/\Gamma$ of 
complex semi-simple Lie groups $G$ do not admit {\it torsionfree} flat affine 
connections.

Interesting exotic deformations of parallelizable manifolds ${\rm SL}(2, \mathbb{C})/ 
\Gamma$ were constructed by Ghys in~\cite{Gh}.
Those  deformations  are constructed  choosing a group homomorphism 
 $u\,:\, \Gamma \,\longrightarrow\, {\rm SL}(2, \mathbb{C})$ and considering the embedding 
 $ \gamma\,\longmapsto\, (u(\gamma), \gamma)$ of $\Gamma$ into ${\rm SL}(2, \mathbb{C})
\times{\rm SL}(2, \mathbb{C})$ (acting on ${\rm SL}(2, \mathbb{C}$) by left and right translations).
Algebraically, the  action is given by:
 $$(\gamma,m) \,\in\, \Gamma \times {\rm SL}(2,\mathbb{C})\,\longmapsto\, u(\gamma^{-1}) m
\gamma\,\in\, {\rm SL}(2,\mathbb{C}).$$
 
It is proved in \cite{Gh} that, for $u$ close enough to the trivial morphism, $\Gamma$ 
acts properly and freely on ${\rm SL}(2, \mathbb{C})$ such that the quotient $M(u, \Gamma)$ 
is a complex compact manifold (covered by ${\rm SL}(2, \mathbb{C)}$). In general, these 
examples do not admit parallelizable manifolds as finite covers. Moreover, for generic 
$u$ the space of holomorphic global vector fields is trivial and the holomorphic 
tangent bundle is {\it simple} (see definition in Section~\ref{section: tensors}).
Notice that the holomorphic flat affine connection on ${\rm SL}(2,\mathbb{C})$ for which 
right invariant vector fields are parallel is also left invariant (since left 
translations preserve parallel vector fields). So this connection descends to the 
quotients $M(u, \Gamma)$.

But the quotients $M(u, \Gamma)$ also admit non-flat holomorphic affine connections. 
In order to see that, consider the (non-degenerate) Killing quadratic form on the Lie 
algebra of ${\rm SL}(2, \mathbb{C})$ and the associated right invariant holomorphic 
Riemannian metric $g$ on ${\rm SL}(2, \mathbb{C})$ (see its definition in
Section~\ref{section: geometric structures}). Since the Killing quadratic form is
invariant under the adjoint 
representation, the induced holomorphic Riemannian metric on ${\rm SL}(2, 
\mathbb{C})$ is bi-invariant, so it descends to the quotients $M(u, \Gamma)$. Notice that
$g$ is locally 
isomorphic to the complexification of the spherical metric on $S^3$ and it has 
constant non-zero sectional curvature. Its associated holomorphic Levi-Civita 
connection is not flat, but {\it locally homogeneous} (since it is induced by a bi-invariant 
connection on ${\rm SL}(2, \mathbb{C})$).

It may be mentioned that a positive answer to the open question at the beginning of the 
introduction implies that compact simply connected manifolds $X$ do not admit holomorphic 
affine connections. Indeed, since $X$ is simply connected, the parallel transport of a 
{\it flat} holomorphic affine connection on $TX$ induces a holomorphic trivialization 
of $TX$. By Wang theorem, $X$ is then biholomorphic to a compact quotient of a complex 
Lie group by a lattice. In particular, the fundamental group of $X$ is infinite.

In this direction, it was recently proved in \cite{DM} that {\it compact complex 
simply connected manifolds of algebraic dimension zero do not admit a holomorphic affine 
connection}.

One of the main results here is a similar theorem for manifolds {\it of 
algebraic dimension one}, assuming that the connection is the Levi-Civita connection 
of a holomorphic Riemannian metric (see Theorem~\ref{theorem: alg dim 1}).

Also Theorem \ref{theorem: simple} proves that {\it compact complex simply connected 
manifolds with simple holomorphic tangent bundle do not admit holomorphic Riemannian 
metrics.}

Another important result is Theorem \ref{theorem: loc hom} which extends the main 
result of \cite{D3} to non-K\"ahler Calabi-Yau manifolds (see Section \ref{section: 
tensors} and~\cite{To}): {\it Holomorphic geometric structures of affine type on 
compact Calabi-Yau manifolds with polystable tangent bundle (with respect to some 
Gauduchon metric) are locally homogeneous}. Notice that this statement is not valid for 
holomorphic geometric
structures of non-affine type, since Ghys constructed in \cite{G2} holomorphic
nonsingular foliations of codimension one on complex tori which are not translation
invariant and also not locally homogeneous.

{}From Theorem \ref{theorem: loc hom} we deduce Corollary \ref{corollary: loc hom} : 
{\it Compact complex simply connected Calabi-Yau manifolds with polystable tangent 
bundle (with respect to some Gauduchon metric) do not admit rigid holomorphic 
geometric structures of affine type.} In particular, they do not admit holomorphic 
affine connections.

The paper is organized in the following way. After this introduction, Section 
\ref{section: geometric structures} introduces the framework of rigid geometric 
structures in Gromov's sense with interesting examples. Section \ref{section: tensors} 
recall the notions of stability and polystability of the holomorphic tangent bundle of 
a complex manifold endowed with a Gauduchon metric and proves the vanishing Lemma 
\ref{lem1} for (non necessarily K\"ahler) Calabi-Yau manifolds. We also deduce here 
Theorem \ref{theorem: loc hom} from Lemma \ref{lem1}. Section \ref{section: alg dim} 
contains Theorem \ref{theorem: algebraic dimension} proving the existence of global 
Killing fields for holomorphic rigid geometric structure on compact complex simply 
connected manifolds of non maximal algebraic dimension. Then we deduce 
Theorem~\ref{theorem: alg dim 1}, Theorem \ref{theorem: simple} and Corollary 
\ref{corollary: loc hom}.

\section{Geometric structures and Killing fields} \label{section: geometric structures}

Holomorphic affine connections are {\it rigid geometric structures} in Gromov's 
sense~\cite{DG}. Let us briefly recall this definition in the holomorphic category.

Let $X$ be a complex manifold of complex dimension $n$. For any integer $k \,\geq\, 1$, we 
associate to it the principal bundle of $k$-frames $R^k(X)\, \longrightarrow
\,X$, which is the bundle of 
$k$-jets of local holomorphic coordinates on $X$. The corresponding structure group 
$D^k$ is the group of $k$-jets of local biholomorphisms of $\mathbb{C}^n$ fixing the 
origin. This $D^k$ is a complex algebraic group.

\begin{definition}
A {\it holomorphic geometric structure} $\phi$ of order $k$ on $X$ is a holomorphic 
$D^k$-equivariant map from $R^k(X)$ to a complex algebraic manifold $Z$ endowed with 
an algebraic action of $D^k$.  The geometric structure $\phi$ is called of {\it affine}
type if $Z$ is a complex affine manifold.
\end{definition}

Holomorphic tensors are holomorphic geometric structures of affine type of order one, 
and holomorphic affine connections are holomorphic geometric structures of affine type 
of order two \cite{DG}. Holomorphic foliations and holomorphic projective connections are 
holomorphic geometric structure of non-affine type.

Another important geometric structure of order one is given in the following definition.

\begin{definition}  A {\it holomorphic Riemannian metric}
on $X$ is a holomorphic section
$$
g\, \in\, H^0(X,\, \text{S}^2((TX)^*))\, ,
$$
where $\text{S}^i$ stands for the $i$-th symmetric product,
such that for every point $x\, \in\, X$ the quadratic  form $g(x)$ on $T_xX$
is nondegenerate. 
\end{definition}

As in the Riemannian or pseudo-Riemannian setting, one associates to a holomorphic 
Riemannian metric $g$ a unique holomorphic affine connection $\nabla$. This connection 
$\nabla$, called the Levi-Civita connection for $g$, is uniquely determined by the 
following properties: $\nabla$ is torsionfree and $g$ is parallel with respect to 
$\nabla$. Using $\nabla$ one can compute the curvature tensor of $g$ which vanishes 
identically if and only if $g$ is locally isomorphic to the standard flat model 
$dz_1^2 + \ldots + dz_n^2$. For more details about the geometry of holomorphic 
Riemannian metrics one can see \cite{D6,DZ}

A more flexible geometric structure is a holomorphic conformal structure.  

\begin{definition} A {\it holomorphic conformal structure} on a complex manifold  $X$
is a holomorphic section $\omega$ of the bundle  $S^2(T^{*}X) \otimes L$, where $L$ is
some holomorphic line bundle over $X$, such that at any point $x$ in $X$ the section
$\omega(x)$ is nondegenerate.
\end{definition} 

Roughly speaking this means that $X$ admits an open cover such that on each open set 
in the cover, $X$ admits a holomorphic Riemannian metric such that on the overlaps of two 
open sets the two given holomorphic Riemannian metrics agree up to a nonzero 
multiplicative constant.

Here the flat example is the quadric \(z_0^2+ z_1 ^2 + \ldots +z_{n+1}^2=0\) in 
$P^{n+1}(\mathbb{C})$ with the conformal structure induced by the quadratic form 
$dz_0^2+dz_1^2+\ldots+dz_{n+1}^2$ on the quadric. The automorphism group of the 
quadric preserving its canonical conformal structure is $\mathbb{P}O(n+2, \mathbb{C})$.
A classical result due to Gauss asserts that all conformal structures on surfaces are 
locally isomorphic to the two-dimensional quadric.
Any complex manifold $M$ of complex dimension $n \,\geq\, 3$ bearing a flat holomorphic
conformal  structure (meaning that the Weyl tensor of curvature vanishes on entire $M$)
is  locally modelled on the quadric.

A holomorphic conformal structure is a holomorphic geometric structure of non-affine 
type.

\begin{definition}
A locally defined holomorphic vector field $Y$ is a (local)
{\it Killing field} of a holomorphic geometric structure
$\phi \,:\, R^k(X) \,\longrightarrow\, Z$ if its canonical lift  from $X$
to $R^k(X)$ preserves the fibers of $\phi$.
\end{definition}

In other words, $Y$ is a Killing field of $\phi$ if and only if its (local) flow 
preserves $\phi$. The Killing vector fields form a Lie algebra with respect to the 
Lie bracket.

A classical result in Riemannian (and pseudo-Riemannian) geometry shows that $Y$ is 
a Killing field of a (holomorphic) Riemannian metric $g$ on $X$ if and only if 
$\nabla_{\cdot} Y$ is a pointwise $g$-skew-symmetric section of $\text{End}(TX)$, 
where $\nabla$ is the Levi-Civita connection associated to $g$~\cite{K}.

A holomorphic geometric structure $\phi$ is called {\it locally homogeneous} if 
the holomorphic tangent bundle $TX$ is spanned by local Killing vector fields of 
$\phi$ in the neighborhood of any point $x\,\in\, X$.  This implies that for any pair 
of points $x\, ,x' \,\in\, X$ there exists a local biholomorphism that
sends $x$ to $x'$ and preserves $\phi$.

A holomorphic geometric structure $\phi$ is {\it rigid} of order $l$ in Gromov's sense 
if any local biholomorphism preserving $\phi$ is determined by its $l$-jet at any 
given point.

Holomorphic affine connections are rigid of order one in Gromov's sense (see the nice 
expository survey \cite{DG}). The rigidity comes from the fact that local 
biholomorphisms fixing a point and preserving a connection linearize in exponential 
coordinates, so they are completely determined by their differential at the fixed 
point.

Holomorphic Riemannian metrics, holomorphic projective connections and holomorphic 
conformal structures in dimension $\geq 3$ are rigid holomorphic geometric structures. 
Holomorphic symplectic structures and holomorphic foliations are non-rigid geometric 
structures~\cite{DG}.

\section{Calabi--Yau manifolds and vanishing results} \label{section: tensors}

Let $X$ be a compact complex manifold. A Hermitian structure on the holomorphic
tangent bundle $TX$ is called a \text{Gauduchon metric} if the corresponding real
$(1\, ,1)$--form $\omega$ on $X$ satisfies the equation
$$
\partial\overline{\partial}\omega^{n-1}\,=\, 0\, ,
$$
where $n\,=\, \dim_{\mathbb C}X$. It is
known that any compact complex manifold admits a Gauduchon metric \cite{Ga}.
Fix a Gauduchon metric on $X$. Let $\omega_0$ be the corresponding $(1\, ,1)$--form
on $X$.

Let $\Omega^{1,1}_{cl}(X,\, {\mathbb R})$ denote the space of all globally
defined $d$--closed real $(1\, ,1)$--forms on $X$. Define the Bott--Chern cohomology for $X$
$$
H^{1,1}_{BC}(X,\, {\mathbb R})\,:=\, \frac{\Omega^{1,1}_{cl}(X,\, {\mathbb R})}{
\{\sqrt{-1}\partial\overline{\partial}\alpha\,\mid\, \alpha\,\in\,
C^{\infty}(X,\,{\mathbb R})\}}\, .
$$
The functional
$$
\Omega^{1,1}_{cl}(X,\, {\mathbb R})\, \longrightarrow\, {\mathbb R}\, ,
~~~~~\alpha\, \longmapsto\, \int_X \alpha\wedge \omega^{n-1}_0
$$
evidently descends to a functional on $H^{1,1}_{BC}(X,\, {\mathbb R})$.

Given a holomorphic line bundle $L$ on $X$, choose a Hermitian structure
$h_L$ on $L$, and consider
the element $c(L)\,\in\, H^{1,1}_{BC}(X,\, {\mathbb R})$ given by the curvature of the
Chern connection on $L$ associated to $h_L$. Clearly, $c(L)$ is independent of the
choice of $h_L$.

For a torsionfree coherent analytic sheaf $F$ on $X$, define
$$
\text{degree}(F)\,:=\, \int_X c(\det F)\wedge \omega^{n-1}_0\, \in\, \mathbb R\, ,
$$
where $\det F$ is the determinant line bundle for
$F$ \cite[p.~166, Proposition~6.10]{K}. The real number
$\text{degree}(F)/\text{rank}(F)$ is called the \textit{slope} of $F$ and it is denoted
by $\mu(F)$. A torsionfree coherent analytic sheaf $V$ on $X$ is called
\textit{stable} if
$$
\mu(F)\, <\, \mu(V)
$$
for all coherent analytic sub-sheaf $F\, \subset\, V$ with $0\, <\, \text{rank}(F)\,
<\, \text{rank}(V)$. If $V$ is a direct sum of stable sheaves of same slope, then it
called \textit{polystable}. See \cite{LT} for more details.

\subsection{Tensors on Calabi--Yau manifolds}

We recall that $X$ is called \textit{Calabi--Yau} if $c(K_X)\,\in\, H^{1,1}_{BC}(X,\, 
{\mathbb R})$ vanishes, where $K_X$ is the canonical line bundle of $X$ \cite{To}. In 
particular, if $K_X$ is trivial or torsion (i.e. of finite order), $X$ is Calabi-Yau, but the converse is not 
true for non-K\"ahler manifolds (see~\cite[Section~3]{To}).

A K\"ahler Calabi--Yau manifold admits
a Ricci flat K\"ahler metric \cite{Ya}, but there is no such result for a general complex
Calabi--Yau manifold. Notice that complex surfaces admitting holomorphic affine
connections have  vanishing first Chern class (in $H^2(X, \mathbb{R})$) \cite{IKO}, but
they are not always 
Calabi--Yau: examples are provided by linear Hopf surfaces. The reader will find more
information and nice examples in~\cite[Section~3]{To}.

\begin{lemma}\label{lem1}
Let $X$ be a compact Calabi--Yau manifold such that
$TX$ is polystable with respect to some Gauduchon metric on $X$. If a holomorphic
section
$$
\psi\, \in\, H^0(X, \, (TX)^{\otimes m}\otimes ((TX)^*)^{\otimes l})
$$
for some $m,l \in \mathbb{N}$ vanishes at some point of $X$, then $\psi\,=\, 0$.
\end{lemma}

\begin{proof}
Since $TX$ is polystable, it admits a Hermitian--Yang--Mills metric $H$
\cite[p. 572]{LY} (see also \cite[p. 61, Theorem 3.0.1]{LT}). The Hermitian
structure on $(TX)^{\otimes m}\otimes ((TX)^*)^{\otimes l}$ induced by $H$ also
satisfies the Hermitian--Yang--Mills condition. Note that this implies that
$(TX)^{\otimes m}\otimes ((TX)^*)^{\otimes l}$ is polystable. 

Since $X$ is Calabi--Yau, the Einstein factor $e_H$ for the Hermitian--Yang--Mills metric $H$
is zero. The Einstein factor for the Hermitian--Yang--Mills metric on
$(TX)^{\otimes m}\otimes ((TX)^*)^{\otimes l}$ induced by $H$ is $(m-l)e_H$, and hence
this Einstein factor is also zero.
Now we conclude that the section $\psi$ in the lemma is flat with respect to the
Chern connection on $(TX)^{\otimes m}\otimes ((TX)^*)^{\otimes l}$ corresponding
to the Hermitian--Yang--Mills structure \cite[p. 50, Theorem 2.2.1]{LT}.
Therefore, if the section $\psi$ vanishes at some  point of $X$ then it vanishes identically.
\end{proof}

\begin{theorem}\label{theorem: loc hom}
Let $X$ be a compact Calabi-Yau manifold such that $TX$ is polystable with respect to 
some Gauduchon metric on $X$. Then any holomorphic geometric structure of affine type 
on $X$ is locally homogeneous.
\end{theorem}

\begin{proof} We make use of Lemma 3.2 in ~\cite[p.~565]{D3} which asserts that
holomorphic geometric structure  are always  locally homogeneous on manifolds  having the property
proved in Lemma~\ref{lem1}: {\it  holomorphic tensor fields  vanishing somewhere are trivial}.
\end{proof} 

A similar result was proved in~\cite{D3} (Theorem 1) for K\"ahler Calabi-Yau manifolds. 

In~\cite{D4} it was exhibited  holomorphic torsionfree affine connections on 
(non-K\"ahler) elliptic bundles $S$ over Riemann surfaces of genus $g \,\geq \,2$ with odd 
first Betti number which are not locally homogeneous. More precisely,  the Killing algebra of a
non-flat connection on $S$ was proved to be of dimension one, generated by the fundamental 
vector field of the elliptic fibration. Moreover the first Chern class of $S$ vanishes 
since $S$ admits flat holomorphic affine connections, but $S$ is not Calabi-Yau.

\subsection{Holomorphic Riemannian metrics and simplicity}

A holomorphic vector bundle $V$ on  a  complex manifold $X$ is called \textit{simple} if
$H^0(X,\, V\times V^*)\,=\, {\mathbb C}\cdot \text{Id}_V$.

\begin{lemma}\label{lem2}
Let $X$ be a   complex manifold such that $TX$ is simple. If
$X$ admits a holomorphic Riemannian metric, then
$$
H^0(X,\,  \Lambda^2(TX)^*)\,=\, 0\, .
$$
\end{lemma}

\begin{proof}
Fix a holomorphic Riemannian metric $g$ on $X$. Let $\theta$ be a globally defined
holomorphic $2$--form on $X$.  Let
$$
\theta'\,=\, H^0(X,\, TX\times (TX)^*)\,=\, H^0(X,\,, \text{End}(TX))
$$
be the endomorphism defined by
$$
g(\theta'(x)(v)\, ,w)\, =\, \theta (x)(v\, ,w)\, , ~~~~~ \forall~~ x\, \in\, X\, ,
~~~~~~~~ v\, ,w\, \in\, T_xX\, .
$$
Since $g$ is fiberwise nondegenerate, this condition uniquely defines $\theta'$.

As $TX$ is simple we have
$\theta'\, =\, \lambda\cdot \text{Id}_{TX}$ for some $\lambda\, \in\, \mathbb C$.
Since $\theta (x)(v\, ,w)\,=\, -\theta (x)(w\, ,v)$, and $g$ is symmetric, it follows that
$$
\lambda\cdot g(v\, ,w) \,=\, g(\theta'(x)(v)\, ,w) \,=\, - g(\theta'(x)(w)\, ,v)\,=\,
-\lambda\cdot g(v\, ,w)\, .
$$
So, $\lambda\,=\, 0$, implying that $\theta\,=\, 0$.
\end{proof}

Notice that compact parallelizable manifolds ${\rm SL}(2, \mathbb{C})/\Gamma$ admit
non-closed holomorphic one-forms induced by the right invariant one-forms on ${\rm SL}(2, 
\mathbb{C})$. Hence Lemma \ref{lem2} is not valid without the hypothesis of simplicity 
of $TX$. Nevertheless, those one-forms do not descend on a generic exotic quotient of 
${\rm SL}(2, \mathbb{C})$ in the sense of Ghys~\cite{Gh}. A generic quotient has simple
holomorphic tangent bundle and hence Lemma~\ref{lem2} applies.

\section{Algebraic dimension and affine connections} \label{section: alg dim}

Recall that the \emph{algebraic dimension} of a complex manifold $X$ is the 
transcendence degree of the field of meromorphic functions of $X$ over the field of 
complex numbers. The algebraic dimension of a projective manifold coincides with its 
complex dimension. In general, the algebraic dimension of a complex manifold 
of dimension $n$ may take any value between $0$ and $n$. Complex manifolds
of maximal algebraic dimension are called Moishezon manifolds.

\begin{theorem}  \label{theorem: algebraic dimension}
Let  $X$  be  a compact, connected and simply connected complex manifold of  complex dimension \(n\) and of algebraic dimension $p$.
Suppose that $X$  admits a holomorphic rigid geometric structure $\phi$. Then $H^0(X, \, TX)$
admits an abelian subalgebra  $A$ with generic orbits of dimension $\geq n-p$ and
preserving $\phi$.
\end{theorem}

\begin{proof} By the main theorem in \cite{D1} (see also~\cite{D2}) the Lie algebra of 
local holomorphic vector fields on $X$ preserving $\phi$ has generic orbits of dimension 
at least $n-p$. Since $X$ is simply connected, by a result due to Nomizu~\cite{No} 
and generalized by Amores~\cite{Am} and then by Gromov~\cite[p.~73, 5.15]{DG}
these 
local vector fields preserving $\phi$ extend to all of $X$. We thus get a finite dimensional 
complex Lie algebra, formed by holomorphic vector fields $v_i$ preserving $\phi$, which 
acts on $X$ with orbits of generic dimension at least $n-p$.

Now put together $\phi$ and the $v_i$ to form another rigid holomorphic geometric 
structure $\phi'\,=\,(\phi,v_i)$ (see \cite{DG} for details on the fact that the 
juxtaposition of a rigid geometric structure with another geometric structure is still 
a rigid geometric structure in Gromov's sense). Considering $\phi'$ instead of $\phi$ and 
repeating the same proof as before, the complex Lie algebra $A$ of those holomorphic 
vector fields preserving $\phi'$ acts on $X$ with generic orbits of dimension at least 
$n-p$. But preserving $\phi'$ means preserving both $\phi$ and the $v_i$. Hence $A$ lies in 
the center of the Lie algebra generated by the $v_i$.  In particular $A$ is a complex 
abelian Lie algebra acting on $X$ with generic orbits of dimension at least $n-p$ and 
preserving $\phi$.
\end{proof}

\subsection{Maximal algebraic dimension}

Assume that $X$ is a compact  Moishezon manifold. This means that the algebraic dimension of
$X$ is $n\,=\, \dim_{\mathbb C}X$.

The following result gives a simple proof of a particular case of Corollary 2 
in~\cite{BM}.

\begin{proposition}\label{prop1}
If $TX$ admits a holomorphic connection, then $X$ admits a finite unramified 
covering by a compact complex torus.
\end{proposition}

The first step of the proof is:

\begin{lemma} \label{lemma: curve} Let $X$ be a complex manifold endowed with an affine holomorphic connection. Then $X$ do not admit nontrivial holomorphic maps from ${\mathbb C}{\mathbb P}^1$ to $X$.
\end{lemma}

\begin{proof}
Let $D$ be a holomorphic connection on $X$. Let
$$
f\, :\, {\mathbb C}{\mathbb P}^1\, \longrightarrow\, X
$$
be a holomorphic map. Consider the pulled back connection $f^*D$ on
$f^*TX$. Note that $\bigwedge^2(T{\mathbb C}{\mathbb P}^1)^*\,=\, 0$ because
$\dim_{\mathbb C} {\mathbb C}{\mathbb P}^1\,=\, 1$. So, the connection $f^*D$ is flat.
Since ${\mathbb C}{\mathbb P}^1$ is simply connected, this implies that
the holomorphic vector bundle $f^*TX$ is trivial.

Now consider the differential of $f$
$$
df\, :\, T{\mathbb C}{\mathbb P}^1\,\longrightarrow\, f^*TX\, .
$$
There is no nonzero
holomorphic homomorphism from $T{\mathbb C}{\mathbb P}^1$ to the trivial holomorphic
line bundle, because $\text{degree}(T{\mathbb C}{\mathbb P}^1)\,> \, 0$. As
$f^*TX$ is trivial, this implies that $df\,=\, 0$. Therefore, $f$ is
a constant map.
\end{proof}

\begin{proof}[Proof of Proposition \ref{prop1}]
Since the Moishezon manifold $X$ does not admit any nonconstant holomorphic map from
${\mathbb C}{\mathbb P}^1$, it is a complex projective manifold \cite[p. 307, Theorem
3.1]{Ca}.

As $TX$ admits a holomorphic connection, $c_i(X, \mathbb{Q})\,=\, 0$ for all $i\, >\, 0$
\cite[p. 192--193, Theorem 4]{At}, where $c_i(X,\mathbb{Q})$ denotes the $i$--th Chern
class of $TX$ with rational coefficients. Therefore, $X$ being complex projective, from Yau's theorem
proving Calabi's
conjecture, \cite{Ya}, it follows that $X$ admits a finite unramified 
covering by a compact complex torus (see also \cite[p. 759, Theoreme 1]{Be} and \cite{IKO}).
\end{proof}

\subsection{Algebraic dimension zero} \label{section: algebraic dimension zero}

We deduce here results about holomorphic geometric structures on manifolds with algebraic dimension zero (compare with Theorem 4.2 in~\cite{D1}). 

\begin{proposition} \label{proposition: zero} Let $X$ be a compact complex manifold of
algebraic dimension zero such that the canonical bundle $K_X$ is of finite order. If $X$ admits
a holomorphic rigid geometric structure, then the fundamental group of $X$ is infinite.
\end{proposition}

\begin{proof}
Assume by contradiction that the fundamental group of $X$ is finite. Then up to a 
finite unramified cover we can assume $X$ simply connected. Since $K_X$ is of finite order
and $X$ is simply connected, it follows that $K_X$ is trivial.

By Theorem~\ref{theorem: algebraic dimension}, we get a complex abelian algebra $A$ 
acting on $X$ with a dense open orbit. Let $Y_1, \cdots, Y_n$ be holomorphic vector 
fields in $A$ which span $TX$ at a generic point. Then $k\,=\, Y_1 \wedge \ldots \wedge 
Y_n$ is a holomorphic section of the anti-canonical bundle $-K_X$. Since $-K_X$ is 
holomorphically trivial, $k$ does not vanish. Hence, the holomorphic vector fields 
$Y_1, \cdots, Y_n$ span $TX$ at every point in $X$ and trivialize the holomorphic 
tangent bundle. The action of the connected complex abelian group  of biholomorphisms associated to the Lie algebra generated by $Y_1, \cdots, Y_n$  is transitive on $X$ 
and hence $X$ is biholomorphic to a  compact  quotient of a connected complex abelian group by a lattice. Consequently, $X$ is 
a complex torus. The fundamental group of $X$ is thus infinite, but this
is a contradiction.
\end{proof}

\begin{corollary} \label{corollary: hol metric} Let $X$ be a compact complex manifold of algebraic dimension zero admitting a holomorphic Riemannian metric. Then the fundamental group of $X$ is infinite.
\end{corollary}

\begin{proof}
The holomorphic Riemannian metric induces a restriction of the structural group of the 
frame bundle $R^1(X)$ of $X$ from ${\rm GL}(n, \mathbb C)$ to $O(n, \mathbb{C})$. Up to a 
double cover we get a restriction to the structural group of $R^1(X)$ to ${\rm SO}(n, 
\mathbb C)$. Hence, up to double cover of $X$, the canonical bundle $K_X$ is 
holomorphically trivial and Proposition \ref{proposition: zero} completes the proof.
\end{proof}

A more general result was  proved in \cite{DM}:

\begin{theorem} Let $X$ be a compact complex manifold with algebraic dimension zero admitting a holomorphic affine connection. Then the fundamental group of $X$ is infinite.
\end{theorem}

\subsection{Algebraic dimension one}

\begin{theorem} \label{theorem: alg dim 1} Let $X$ be a compact complex manifold of algebraic dimension one admitting a holomorphic Riemannian metric $g$. Then the fundamental group of $X$  is infinite.
\end{theorem}

\begin{proof}
Assume by contradiction that $X$ has finite fundamental group. Up to a finite 
unramified cover, we can assume that $X$ is simply connected.

By Theorem \ref{theorem: algebraic dimension} there exists a finite dimensional 
abelian Lie subalgebra $A\, \subset\, H^0(X,\, TX)$ 
preserving $g$ and acting with generic orbits of dimension at least $n-1$, where $n$ is the 
complex dimension of $X$. The case where $A$ acts with an open orbit was settled in 
Section~\ref{section: algebraic dimension zero}.

Let us consider now the case where the generic orbits of $A$ are of dimension $n-1$. 
Since $A$ is abelian, the isotropy of a point $x \in X$ lying in a generic orbit $O$ 
acts trivially on $T_xO$ and hence on $T_xX$ (an orthogonal linear map which is identity on a 
hyperplane is trivial). Consequently, the isotropy of $A$ at $x$ is trivial and $A$ 
must have dimension $n-1$.
Let $$Y_1, Y_2, \cdots, Y_{n-1}\,\in\, H^0(X,\, TX)$$ be a basis of $A$.

Denote also by 
$g$ the symmetric bilinear form associated to the quadratic form $g$.
The functions $g(Y_i,Y_j)$ are holomorphic and hence are constant on $X$. So
$$
V \cdot g(Y_i,Y_j)\,=\,0
$$
for any local holomorphic vector field $V$, so
\begin{equation}\label{f2}
g(\nabla_VY_i, 
Y_j)+ g(Y_i,\nabla_V Y_j)\,=\,0\, .
\end{equation}

Since $Y_i$ are Killing vector fields with respect to $g$, it follows that 
$$
\varphi_i\, :\, TX\, \longrightarrow\, TX\, , ~~~~~~~~
w\, \longmapsto\, \nabla_w Y_i
$$
are pointwise $g$-skew-symmetric. Therefore, we have 
\begin{equation}\label{f1}
g(\nabla_V Y_i,\, W)+g(\nabla_W{Y_i},\, V)\,=\, 0
\end{equation}
for all locally defined holomorphic vector fields $V$, $W$. Combining \eqref{f2}
and \eqref{f1},
$$
0\,=\, g(\nabla_VY_i, Y_j)+ g(Y_i,\nabla_V Y_j)\,=\,- g(\nabla_{Y_j}Y_i, V)
-g(V,\nabla_{Y_i} Y_j)\,=\, -g(\nabla_{Y_j}Y_i+\nabla_{Y_i} Y_j, V)\, ,
$$
so $\nabla_{Y_j}Y_i+\nabla_{Y_i} Y_j\,=\, 0$. On the other hand,
the Levi-Civita connection $\nabla$ is torsionfree and the $Y_i$ commute, so
$\nabla_{Y_j}Y_i\,=\,\nabla_{Y_i}Y_j$. Therefore, we have
\begin{equation}\label{f3}
\nabla_{Y_j}Y_i\,=\, 0
\end{equation}
for all $i,j \,\in \, \{1, \cdots, n-1 \}$. Now combining \eqref{f2}, \eqref{f1} and
\eqref{f3},
$$
0\,=\,g(\nabla_VY_i, Y_j)+ g(\nabla_VY_j,Y_i)\,=\,g(\nabla_VY_i, Y_j)- 
g(\nabla_{Y_i}Y_j,V)\,=\,g(\nabla_VY_i, Y_j)
$$
for any local holomorphic vector field $V$ and $i,j \,\in \, \{1, \cdots, n-1 \}$.
So, the image $\varphi_i(TX)$ lies in the orthogonal part
$\{Y_1, \cdots, Y_{n-1}\}^{\perp}$ for all $i \,\in\, \{1, \cdots, n-1 \}.$

Fix a point $x \,\in\, X$ lying on a generic orbit of $A$. Then 
$\dim \varphi_i(T_xX)\, \leq\, 1$ is at most one as it lies
in $\{Y_1(x), \cdots, Y_{n-1}(x)\}^{\perp}$.
On the other hand, the rank of a skew-symmetric $g$-endomorphism is always even. It
now follows that $\varphi_i$ vanishes identically for all $i$.

Since $\varphi_i\,=\, 0$, the one-form $\omega_i$ dual to $Y_i$, defined by
$v\, \longmapsto\, g(Y_i,v)$, is closed. Indeed, if $V_1$ and $V_2$ are local holomorphic
vector fields then
$$
d \omega_i (V_1, V_2)\,=\,V_1 \cdot \omega_i (V_2) - V_2 \cdot \omega_i(V_1)-
\omega_i (\lbrack V_1, V_2 \rbrack)\,=\, g(\nabla_{V_1}Y_i, V_2) + g(Y_i, \nabla_{V_1}V_2)
$$
$$
-g(Y_i, \nabla_{V_2}V_1)-g(\nabla_{V_2}Y_i, V_1)-g(Y_i, \lbrack V_1, V_2 \rbrack)\,=\,
g(Y_i, \nabla_{V_1}V_2-\nabla_{V_2}V_1-\lbrack V_1, V_2 \rbrack)
$$
$$
=\, g(Y_i, T(V_1,V_2))\,=\, g(Y_i,0)\,=\,0
$$
where $T$ is the torsion of the Levi-Civita $\nabla$ connection (which is zero).
The above computation shows that the dual form $v\, \longmapsto\, g(Y,v)$ is closed if and
only if $v \, \longmapsto\,\nabla_v Y$ is pointwise $g$-symmetric. In particular, if $Y$ is a Killing
field with respect to $g$, the associated dual 
form $v\, \longmapsto\, g(Y,v)$ is closed if and only if $Y$ is parallel.

Since $X$ is simply connected, the closed form $\omega_i$ must be exact. This implies
$\omega_i$ vanishes identically because there is no nonconstant holomorphic function
on $X$. Hence $Y_i$ vanishes identically, which is a contradiction.
\end{proof}

\subsection{Other applications}

Theorem \ref{theorem: loc hom} has the following corollaries.

\begin{corollary}\label{corollary: loc hom} Let $X$ be a compact  Calabi-Yau manifold such that  $TX$ is polystable with respect to some Gauduchon metric on $X$. If $X$ admits  a  rigid holomorphic geometric  structure of affine type $\phi$, then the fundamental group of $X$ is infinite.
\end{corollary}

\begin{proof}
{}From Theorem~\ref{theorem: loc hom} we know that $\phi$ is locally homogeneous. Consider
local holomorphic Killing vector fields $Y_1, \cdots, Y_n$ which span $TX$ on an open subset
in $X$.

Assume by contradiction that the fundamental group is finite. Then, up to a finite 
unramified cover, $X$ is simply connected and Nomizu's extension result~\cite{No} 
proves that the Killing vector fields $Y_i$ extend to global holomorphic vector fields 
$\widetilde{Y_i} \,\in\, H^0(X, \, TX)$. Moreover, Lemma~\ref{lem1} shows that the holomorphic 
tensor $\widetilde{Y_1} \wedge \cdots \wedge \widetilde{Y_n}$ do not vanish on $X$. Hence, the 
$\{\widetilde{Y_i}\}$ trivialize the holomorphic tangent bundle of $X$. By a result of 
Wang~\cite{Wa}, $X$ is biholomorphic to a quotient of a complex Lie group $G$ by a 
lattice. In particular, the fundamental group of $X$ is infinite, which is a contradiction. 
\end{proof}

\begin{corollary}\label{corollary: Riem metric}
Let $X$ be a compact complex manifold such that $TX$ is polystable with respect to 
some Gauduchon metric on $X$. If $X$ admits a holomorphic Riemannian metric $g$, then 
$g$ is locally homogeneous and the fundamental group of $X$ is infinite.
\end{corollary}

\begin{proof} As in the proof of Corollary~\ref{corollary: hol metric}, the manifold
$X$ admits a double cover with trivial canonical bundle. Hence, $X$ is Calabi-Yau and Theorem \ref{theorem: loc hom} and 
Corollary~\ref{corollary: loc hom} apply.
\end{proof}

\begin{corollary} Let $X$ be a compact complex manifold such that the canonical bundle $K_X$ is of finite order  and  $TX$ is polystable with respect to some Gauduchon metric on $X$. Then
any holomorphic conformal structure and any holomorphic projective connection on $X$ are locally homogeneous. In particular, if $X$ admits a holomorphic conformal structure or a holomorphic projective connection then the fundamental group of $X$ is infinite.
\end{corollary}

\begin{proof}
Up to a finite cover, $K_X$ is trivial. Then any holomorphic conformal structure 
admits a global representative which is a holomorphic Riemannian metric and 
Corollary~\ref{corollary: Riem metric} applies. Also, since $K_X$ is trivial, any 
holomorphic projective connection admits a global representative which is a 
holomorphic affine connection (see \cite[p.~96]{Gu} and \cite[p.~78--79]{KO}) and 
Corollary~\ref{corollary: loc hom} applies.
\end{proof}

\begin{theorem} \label{theorem: simple} Let $X$ be a compact complex manifold such that $TX$ is simple. If $X$ admits a holomorphic Riemannian metric, then the fundamental group of $X$ is infinite.
\end{theorem}

\begin{proof}
Since $TX$ admits a holomorphic affine connection, if $X$ is Moishezon, then 
Proposition~\ref{prop1} shows that $X$ has a finite  cover which is a  torus. In
particular the fundamental group of $X$ is infinite.

Consider  now the case where the algebraic dimension  of $X$ is not maximal. 

Assume, by contradiction, that the fundamental group of $X$ is finite. Up to a finite 
cover, we suppose that $X$ is simply connected. Then Theorem \ref{theorem: algebraic 
dimension} implies that there exists a nontrivial globally defined Killing field $Y$ 
on $X$. Its dual $\omega$, $v \, \longmapsto\,g(Y, v)$, is a holomorphic one-form on $X$. 
Lemma~\ref{lem2} implies that $d \omega \,=\, 0$. Since $X$ is simply connected, $\omega$
must be exact and so $\omega \,=\,0$ because there is no nonconstant holomorphic function
on $X$: a contradiction (since $Y$ is nontrivial).
\end{proof}

\subsection{Nef tangent bundle}

Let $X$ be a compact complex manifold such that the holomorphic tangent bundle $TX$ is 
nef \cite[p.~305, Definition~1.9]{DPS} (see also \cite[p.~299, Definition~1.2]{DPS}). 
Assume that $X$ admits a holomorphic Riemannian metric $g$. Then $g$ identifies 
$TX$ with $(TX)^*$, so $TX$ is numerically flat \cite[p.~311, Definition~1.17]{DPS}. 
Hence $(TX)^{\otimes m}\otimes ((TX)^*)^l$ is numerically flat for all nonnegative 
integers $m, l$ \cite[p.~307, Proposition~1.14]{DPS}. Hence any holomorphic section
of $(TX)^{\otimes m}\otimes ((TX)^*)^l$ that vanishes at some point of $X$ vanishes
identically \cite[p.~310, Proposition~1.16]{DPS}. Therefore, 
any holomorphic geometric structure of affine type on $X$ is locally homogeneous;
see the proof of Theorem \ref{theorem: loc hom}. Now we conclude that
$g$ is locally homogeneous and the fundamental group of $X$ is infinite (see
the proof of Corollary \ref{corollary: Riem metric}).

\section*{Acknowledgements}

We thank Valentino Tosatti for helpful comments. The first-named author wishes
to thank the Laboratoire J. -A. Dieudonn\'e  for hospitality. 

%%%%%%%%%%%%%%%%%%%%%%%%%%%%%%%%%%%%%%%%%%%%%%%%%%%%%%%%%%%%%%

\end{document}